\documentclass{amsart}

\usepackage{url}
\usepackage{color}

\usepackage[backref=page,colorlinks,bookmarksopen=true]{hyperref}
\usepackage{amsfonts,amsmath,amssymb}
\usepackage{amsthm}
\usepackage{enumerate}
\usepackage{appendix}
\usepackage{pgf,tikz}
\usetikzlibrary{arrows}

\usepackage{mathrsfs, xspace}
\usepackage{stmaryrd}
\usepackage{graphicx,psfrag}
\usepackage{a4wide}
\usepackage[all]{xy}
\theoremstyle{plain}
\newtheorem{proposition}{Proposition}[section]
\newtheorem{theorem}[proposition]{Theorem}

\newtheorem{lemma}[proposition]{Lemma}
\newtheorem{corollary}[proposition]{Corollary}

\theoremstyle{definition}
\newtheorem{definition}[proposition]{Definition}

\theoremstyle{remark}
\newtheorem{remark}[proposition]{Remark}

\everymath{\displaystyle}

\DeclareMathOperator{\Conv}{Conv}

\DeclareMathOperator{\Isom}{Isom}

\newcommand{\bd}{\partial}

\newcommand{\action}{\curvearrowright}

\newcommand{\R}{\mathbb{R}}                          % |R
\newcommand{\N}{\mathbb{N}}                          % |N
\DeclareMathOperator{\dd}{d\!}

\newcommand{\SG}{\mathcal{S}(G)}

\begin{document}
\title[Geometric density for IRS]{Geometric Density for invariant random subgroups\\ of groups acting on CAT(0) spaces}
\author{Bruno Duchesne, Yair Glasner, Nir Lazarovich, Jean L\'ecureux}
\date{\today}

\begin{abstract}  We prove that an IRS of a group with a geometrically dense action on a CAT(0) space also acts geometrically densely; assuming the space is either of finite telescopic dimension or locally compact with finite dimensional Tits boundary. This can be thought of as a Borel density theorem for IRSs.
\end{abstract} 
\maketitle

\section{Introduction}

\subsection{Invariant random subgroups}Let $G$ be a locally compact second countable group. We denote by $\mathcal{S}(G)$ the space of closed subgroups of $G$. We endow it with the Chabauty topology. With this topology $\SG$ is a compact metrizable space \cite[Propositions 1.7\&1.8]{MR2406240}. Recently, a new and fruitful point of view about non-free probability measure preserving (shortly p.m.p) actions has appeared (see \cite{AGV:Kesten_IRS}) and is currently a fast growing field of research.

\begin{definition}An \emph{invariant random subgroup} (shortly IRS) is a Borel probability measure on $\SG$ which is invariant under the adjoint action of $G$ on $\SG$ by conjugations. 
\end{definition}

We say that an IRS is \emph{not trivial}, if almost surely it is not the trivial group (i.e. $\mu(\{e\})=0)$). We emphasize that this convention is not widely used (for example it is different from the one in \cite{Tucker-Drob:2012vn}) but it is more convenient for us and there will be no ambiguity in case of an ergodic IRS.

\subsection{Geometric density}

In \cite{MR2574741}, P.-E. Caprace and N. Monod proved a density theorem in the spirit of Borel density theorem for groups acting  on CAT(0) spaces. Let $X$ be a CAT(0) space.

\begin{definition}
The action of a subgroup $H<\Isom(X)$ is called \emph{minimal} if  it does not stabilize a strict closed convex subset of $X$. It is called \emph{geometrically dense} if it is minimal and if $H$ does not fix a point in $\bd X$.
\end{definition}

Caprace and Monod proved that if $G$ acts continuously and geometrically densely on a proper CAT(0) space with finite dimensional Tits boundary, then the same holds for  closed subgroups of finite covolume \cite[Theorem 2.4]{MR2574741}. This result implies Borel density in case $G$ is a semi-simple algebraic group over a local field without anisotropic factor \cite[Proposition 2.8]{MR2574741}. They also proved, under the same assumptions, that every action of a normal subgroup is, too, geometrically dense \cite[Theorem 1.10]{MR2574740}. Since IRSs generalize both subgroups of finite covolume and normal subgroups, our theorem below subsumes both theorems of Caprace-Monod.

\begin{theorem}\label{IRSdense}Assume $X$ has finite telescopic dimension or is proper with finite dimensional Tits boundary, is irreducible and not the real line.
If $G$ acts faithfully, continuously and geometrically densely by isometries on $X$ then every non trivial IRS of $G$ is also geometrically dense.
\end{theorem}

In case, the space is not irreducible, one may also obtain geometric density under the assumption that the IRS does not act trivially on any irreducible factor. This hypothesis is in particular weaker than the hypothesis of irreducibility (and non triviality) that already appeared in \cite[\S4]{Abert:2012qy}. Observe that in case $n=1$, Theorem \ref{product case} reduces to Theorem \ref{IRSdense}.

%\begin{definition}Assume $G$ is a product of l.c.s.c. groups $G=G_1\times\dots\times G_n$. An IRS is irreducible if the action of any factor $G_i\action\SG$ is ergodic.\end{definition}

\begin{theorem}\label{product case}Let $G_1,\dots,G_n$ be l.c.s.c groups such that each one acts continuously faithfully geometrically densely by isometries on some CAT(0) space $X_i$ which is irreducible ($\neq \R$) with finite telescopic dimension or locally compact with finite dimensional Tits boundary.  Let $G=G_1\times\dots\times G_n$ and $X=X_1\times\dots\times X_n$. Let $\mu$ be an IRS of $G$ and $\mu_i$ be the pushforward of $\mu$ by the projection $G\to G_i$. Assume that for every $i$, the IRS $\mu_i$ of $G_i$ is not trivial. Then $\mu$-almost every $H\in\SG$ acts geometrically densely on $X$.
\end{theorem}

In this particular situation (when $G$ acts geometrically densely on some nice CAT(0) space), one recovers  the result \cite{BDL} that an amenable IRS lies in the amenable radical, which is trivial in this situation. The paper \cite{BDL} was actually obtained after proving the following.

\begin{corollary}\label{amenable}
With the same assumptions as in Theorem \ref{product case}, any IRS of $G$ with non trivial projections has a trivial amenable radical.\end{corollary}

\begin{proof} Since $X$ has a trivial Euclidean factor, there is no amenable group acting geometrically densely on $X$ \cite[Theorem 1.6]{MR2558883}. 
\end{proof}

%After writing this note, two of the authors realized the argument could be adapted to generalize Corollary \ref{amenable} to the general case. This will appear in a subsequent paper \cite{BDL}

\subsection{IRSs in linear groups and Borel density theorem} 
The paper \cite{Gelander:2014kq} initiates a systemic study of IRSs in linear groups. In particular a Borel density theorem for IRS in countable linear groups (over any field) was obtained \cite[Theorem A.1]{Gelander:2014kq}. As a corollary of Theorem \ref{product case}, we get a Borel density theorem for IRSs in semisimple groups over local fields. Such a theorem was known by experts, but we include it since it is an easy applications of our previous theorem.

Let us begin with  definitions about semisimple algebraic groups and their subgroups. Let $k$ be a local field and assume $ G$ is the $k$-points of a connected semisimple $k$-group ${\bf G}$. There is an adjoint $k$-group $\overline{\bf{G}}$ which decomposes as $\mathbf{G}_1\times\dots\times {\mathbf{G}}_n$ where each $\mathbf{G}_i$ is an adjoint simple $k$-group and there is a $k$-isogeny $\overline{p}\colon \mathbf{G}\to \overline{\bf{G}}$. Moreover a subgroup of $H\leq G$ is Zariski-dense as soon as $\pi(H)$ is Zariski-dense.  

\begin{definition}Let $H$ be a subgroup of $G$. We say that $H$ \emph{has non-trivial projections} if for all $i\in\{1,\dots,n\}$, $\pi_i(\overline{p}(H))\neq\{e\}$ where $\pi_i$ is the projection $\mathbf{G}\to\mathbf{G}_i$. An IRS of $G$ has non-trivial projections if almost surely, any $H\in\SG$ has non-trivial projections.\end{definition}

\begin{theorem}\label{Zariski density}Let $k$ be a local field and let ${G}$ be  the $k$-points of a $k$-isotropic semisimple algebraic $k$-group $\mathbf{G}$. Then any IRS of $G$ with non-trivial projections is Zariski-dense.
\end{theorem}

\begin{proof}The group $\overline G=\overline{\mathbf{G}}(k)$ acts faithfully and geometrically densely on its Bruhat-Tits building (respectively symmetric space of non-compact type in the archimedean case) $X$ which a product of the (irreducible $\neq\R$) Bruhat Tits buildings (respectively symmetric space) of the groups factors $G_i$'s. The image in $\overline{G}$ of an IRS with non-trivial projections satisfies the hypothesis of Theorem \ref{product case}. Thus almost surely $H\in\mathcal{S}(\overline{G})$ acts geometrically densely and \cite[Proposition 2.8]{MR2574741} shows that $H$ is Zariski-dense.
\end{proof}

\begin{remark} Nevo-Stuck\&Zimmer theorem (formulated in terms of IRSs see \cite[Theorem 4.1]{Abert:2012qy}) shows that non-atomic irreducible IRSs of higher rank  in semi-simple Lie groups with property (T) and no center are irreducible lattices and thus are Zariski-dense (see merely \cite[Theorem 2.6]{Abert:2012qy} in the simple case). So Theorem \ref{Zariski density} is mostly useful for non-archimedean fields with valuation.
\end{remark}

\section{Actions of IRSs on CAT(0) spaces}

\subsection{Actions on CAT(0) spaces}
In this subsection, we recall the required facts about CAT(0) spaces. 

We fix a CAT(0) space $X$. A large part of what will follow is inspired by \cite{MR2574741,MR2574740}, and we refer to these references for details. We assume either that $X$ is proper  with finite-dimensional boundary, or that $X$ is of finite telescopic dimension (see \cite{MR2558883} for the definition of the telescopic dimension).
%[May we replace those assumptions by ``$\overline{X}$ is $\mathscr{T}_c$-compact and $\partial X$ has finite dimension'' ? Proposition \ref{affine} will be hard to obtain.***].
These geometric assumptions are used to guarantee the following fact: any  decreasing sequence of closed convex subsets has an non-empty intersection or defines a canonical point at infinity. We also assume $X$ to be separable. There is no loss of generality with this last assumption since a second countable group acting continuously by isometries on a CAT(0) space has a separable closed subset which is invariant (for example, the closed convex hull of some orbit).\\

We  use \emph{measurable fields of CAT(0) spaces} as a convenient tool to deal with measurability questions. We refer to \cite{MR3044451,MR3163023} for generalities about such fields. One does not need a deep knowledge of them to understand our proof. It suffices to have the following situation in mind. The group $G$  acts continuously by isometries on $X$. To each subgroup $H\in \SG$ we  associate a closed convex subset $X_H$. It will vary measurably with $H$ in sense that for any $x\in X$, the distance function $H\mapsto d(x,X_H)$ is measurable. The collection ${\bf X}=\{X_H\}_{H\in\SG}$ is called a \emph{subfield} of the constant field ${\bf X}_0$ with fiber $X$ over the measurable space $\SG$. The group $G$  acts on $\bf X$, meaning that  for almost all $H\in\SG$ and for all $g\in G$, $gX_H=X_{gHg^{-1}}$. We also say that $\bf X$ is $G$-invariant. In particular, $G$-invariance implies that for almost all $H\in\SG$, $X_H$ is a closed convex $H$-invariant subspace of $X$.

%The notion of \emph{sections} of a field will be useful. A section ${\bf x}$ of the field $\bf X$ is collection $(x_H)$ such that for all $H\in\SG$, $x_H\in X_H$ and for all $x\in X$, $H\mapsto d(x,x_H)$ is measurable. A \emph{fundamental family} of $\bf X$ is a countable collection $({\bf x}^n)$ of sections of $\bf X$ such that almost surely $(x^n_H)_{n\in\N}$ is dense in $X_H$.

%For groups acting by isometries on $X$, one can associate a canonical convex set at infinity.
%
%
%\begin{definition}[{\cite[Lemma 4.2]{MR2574740}}]
%Let $H<\Isom(X)$. The \emph{convex limit set} $\Delta_H$ of $H$ is defined as $\bd \Conv(H.x)$, for some $x\in X$. It does not depend on the choice of $x$.
%\end{definition}
%
%
%\begin{definition}
%A closed convex subset $Y\subset X$ is called \emph{boundary-minimal} if for every closed convex subset $Z\subsetneq Y$, we have $\bd Z\subsetneq \bd Y$.
%\end{definition}
%
%\begin{theorem}\label{CDelta}
%Let $\Delta\subset \bd X$. Define 
%$$\mathcal C_\Delta=\{ Y\subset X \mid Y \textrm{ is boundary minimal and } \bd Y=\Delta\}.$$
%
%Then $\bigcup_{Y\in \mathcal C_\Delta} Y$ is a closed convex subset which is isometric to $Y\times C$, for some $Y\in \mathcal C_\Delta$. Furthermore, every element of $\mathcal C_\Delta$ is equal to some $Y\times\{c\}$. 
%\end{theorem}
%
%\begin{proof}This theorem is proved in 
%\cite[Prop 3.6]{MR2574740} for proper CAT(0) spaces. The same proof applies in the case of finite telescopic dimension.
%\end{proof}
%
\subsection{Proof of Theorems \ref{IRSdense} \& \ref{product case}}

For the remainder of this subsection, we assume that $X$ has trivial Euclidean factor and that $G$ acts faithfully, continuously and geometrically densely by isometries on $X$. 

Let $(S,\mu)$ be a standard probability space with a p.m.p action of $G$. Let $x_0$ be an arbitrary point in $X$ and let $\mathcal{C}_0$ be the space of 1-Lipschitz convex functions vanishing at $x_0$.

\begin{proposition}\label{affine}Let $s\mapsto f_s$ be a map from $S$ to $\mathcal{C}_0$ such that for all $x\in X$, the map $s\mapsto f_s(x)$ is measurable and for all $x\in X,\ g\in G$ and almost all $s\in S$, $f_s(gx)=f_{g^{-1}s}(x)+f_s(gx_0)$. Then almost surely $f_s$ is constant. 
\end{proposition}

\begin{proof}Define 
\begin{equation}\label{r}f(x)=\int_{S}f_s(x)\dd\mu(s).
\end{equation}
First observe that $|f_s(x)|\leq d(x_0,x)$ which shows that the right-hand side of \eqref{r} is well-defined. The function $f$ is a 1-Lipschitz convex function on $X$. It is moreover quasi-invariant: $f(gx)=f(x)+f(gx_0)$ for all $g\in G$, $x\in X$. 

 If $f$ does not achieve its minimum then  $\left(f^{-1}(-\infty,r])\right)_{r>\inf f}$ yields filtering family of closed convex subsets with empty intersection. Thus  \cite[Lemma 5.5]{MR2558883} or  \cite[Proposition 3.2]{MR2574740} yields a canonical point $\xi\in\partial X$ which is $G$-invariant since $gf^{-1}\left((-\infty,r]\right)=f^{-1}\left((-\infty,r+f(gx_0)]\right)$.
Thus $f$ achieves its minimum. Quasi-invariance implies that $g\mapsto f(gx_0)$ is a homomorphism, which is trivial since $f$ has a minimum. Thus $f$ is $G$-invariant and the set of points where this minimum is achieved is $G$-invariant. By minimality it is $X$. Thus $f$ is constant. Since $f_s$ is continuous, and $X$ is separable, this implies that almost all $f_s$ are affine but triviality of the Euclidean factor of $X$ and \cite[Proposition 4.8]{MR2558883} (for the finite telescopic dimension case) or  \cite[Corollary 1.8]{MR2262730} (for the proper case) imply almost all $f_s$ are constant.
\end{proof}

\begin{lemma}\label{convexfield}Assume there is a map $s\mapsto X_s$ from $S$ to the set of closed convex subsets of $X$ which is $G$-equivariant and such that for any $x\in X$, $s\mapsto d(x,X_s)$ is measurable. Then for almost every $s$, $X_s=X$.
\end{lemma}

\begin{proof}
It suffices to apply Proposition \ref{affine} to $f_s=d(x,X_s)-d(x_0,X_s)$.\end{proof}

\begin{lemma}\label{boundary map} There is no $G$-equivariant measurable map $S\to \partial X$.
\end{lemma}

\begin{proof}Let $s\mapsto \xi_s$ be such a map. We denote by $f_s(x)$ the Busemmann function associated to $\xi_s$ vanishing at $x_0$. Proposition \ref{affine} implies that almost all $f_s$ are constant but constant Busemann functions do not exist.
\end{proof}

For $H\in\SG$ and $x\in X$, we denote by $\mathscr{C}_H(x)$ the closed convex hull of the $H$-orbit of $x$. A subset is $H$\emph{-minimal} if it is closed, convex, $H$-invariant and minimal among closed convex (non empty) $H$-invariant subsets. Recall that any two $H$-minimal subsets are parallel, in particular isometric to some subspace $M_H$ and the union of all such minimal subspaces $Z_H$ splits as a product $Z_H\simeq M_H\times T_H$ \cite[Theorem 4.3]{MR2574740}.

Let $C\subseteq X$ be a closed, convex and $H$-invariant subspace. It minimal if and only if $C=\mathscr{C}_H(x)$ for all $x\in C$. Or in other words, it is minimal if and only if for all $x,y\in C$, $d(x,\mathscr{C}_H(y))=0$. Thus, in order to recover, measurably, $H$-minimal subspaces, we define
\[\varphi_H(x)=\sup_{y\in\mathscr{C}_H(x)}d(x,\mathscr{C}_H(y)).\]
It follows that $x\in X$ belongs to an $H$-minimal subspace if and only if $\varphi_H(x)=0$. Moreover it will be shown in Lemma \ref{varphimin} that $X$ admits an $H$-minimal subspace if $\varphi_H(x)$ is finite for some $x \in X$ (equivalently for all $x \in X$). Actually, we will see that $d(x,Z_H) \leq \varphi_H(x)\leq 2d(x,Z_H)$. Thus $\varphi_H$ gives us a quantitative tool to express how far $\mathscr{C}_H(x)$ is from being $H$-minimal.

\begin{remark} We will not use this fact, but one may observe that $\varphi_H$ is actually a convex continuous function. Furthermore, in case $\mathscr{C}_H(x)$ contains a unique $H$-minimal subspace then $\varphi_H(x)= d(x,Z_H)$. This equality is not always satisfied, as the following example shows. Let $T_3$ be the regular trivalent tree and $X=T_3\times T_3$ considered as a CAT(0) square complex. Let $D$ be the diagonal of some square and $H$ be the pointwise stabilizer of $D$. For  $x \in X$, let $p(x)$ be its projection. One has $p(hx)=hp(x)$  for all $h\in H$. That is $p(hx)=p(x)$. Now if $x$ and $y$ have same projection on $D$, it does not mean that the midpoint $m$ of $[x,y]$ projects on $p(x)=p(y)$ as the following drawing of a small part of $X$ illustrates. 

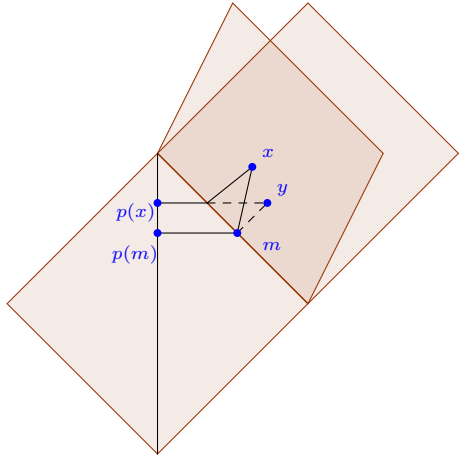
\begin{figure}[h]
\definecolor{qqqqff}{rgb}{0,0,1}
\definecolor{zzttqq}{rgb}{0.6,0.2,0}
\begin{tikzpicture}[line cap=round,line join=round,>=triangle 45]
\clip(-2.09,-2.26) rectangle (4.15,4.07);
\fill[color=zzttqq,fill=zzttqq,fill opacity=0.1] (0,2) -- (2,0) -- (0,-2) -- (-2,0) -- cycle;
\fill[color=zzttqq,fill=zzttqq,fill opacity=0.1] (0,2) -- (2,4) -- (4,2) -- (2,0) -- cycle;
\fill[color=zzttqq,fill=zzttqq,fill opacity=0.1] (0,2) -- (1,4) -- (3,2) -- (2,0) -- cycle;
\draw [color=zzttqq] (0,2)-- (2,0);
\draw [color=zzttqq] (2,0)-- (0,-2);
\draw [color=zzttqq] (0,-2)-- (-2,0);
\draw [color=zzttqq] (-2,0)-- (0,2);
\draw [color=zzttqq] (0,2)-- (2,4);
\draw [color=zzttqq] (2,4)-- (4,2);
\draw [color=zzttqq] (4,2)-- (2,0);
\draw [color=zzttqq] (2,0)-- (0,2);
\draw [color=zzttqq] (0,2)-- (1,4);
\draw [color=zzttqq] (1,4)-- (3,2);
\draw [color=zzttqq] (3,2)-- (2,0);
\draw [color=zzttqq] (2,0)-- (0,2);
\draw (0,2)-- (0,-2);
\draw (0.66,1.34)-- (1.26,1.82);
\draw (1.26,1.82)-- (1.06,0.94);
\draw (0,0.94)-- (1.06,0.94);
\draw [dash pattern=on 3pt off 3pt] (0.66,1.34)-- (1.46,1.34);
\draw (0,1.34)-- (0.66,1.34);
\draw [dash pattern=on 3pt off 3pt] (1.06,0.94)-- (1.46,1.34);
\begin{scriptsize}
\fill [color=qqqqff] (0,1.34) circle (1.5pt);
\draw[color=qqqqff] (-0.29,1.22) node {$p(x)$};
\fill [color=qqqqff] (1.46,1.34) circle (1.5pt);
\draw[color=qqqqff] (1.66,1.52) node {$y$};
\fill [color=qqqqff] (1.26,1.82) circle (1.5pt);
\draw[color=qqqqff] (1.47,2.01) node {$x$};
\fill [color=qqqqff] (1.06,0.94) circle (1.5pt);
\draw[color=qqqqff] (1.52,0.77) node {$m$};
\fill [color=qqqqff] (0,0.94) circle (1.5pt);
\draw[color=qqqqff] (-0.3,0.66) node {$p(m)$};
\end{scriptsize}
\end{tikzpicture}
\end{figure}
\end{remark}

For any closed $Y,Z\subseteq X$ we denote by $D(Y,Z)$ the Hausdorff distance (possibly infinite) between $Y$ and $Z$.

\begin{lemma}\label{dist}For any $x,y\in X$, $D(\mathscr{C}_H(x),\mathscr{C}_H(y))\leq d(x,y)$.
\end{lemma}

\begin{proof} Since $d(hx,hy)=d(x,y)$ for any $h\in H$, the $H$-orbit of $y$ lies in the closed $d(x,y)$-neighborhood of $\mathscr{C}_H(x)$. This neighborhood is convex and thus contains $\mathscr{C}_H(y)$. Exchanging $x$ and $y$, one obtains the result.
\end{proof}

\begin{lemma}\label{varphimin} For any $H\in\SG$ the following are equivalent:
\begin{enumerate}[(i)]
\item  the subset $Z_H$ is not empty,
\item for any $x\in X$, $Z_H\cap\mathcal{C}_H(x)\neq\emptyset$,
\item there is $x\in X$ such that $\varphi_H(x)<\infty$,
\item for all $x\in X$,  $\varphi_H(x)<\infty$.
\end{enumerate}
\end{lemma}

\begin{proof}Clearly (ii) implies (i) and (iv) implies (iii).

$\bf (i)\Rightarrow (iv)$. Assume $M$ is an $H$-minimal subset. Since $M=\mathcal{C}_H(m)$ for any $m\in M$, Lemma \ref{dist} implies that for all $x\in X$, $d(x,M)\geq d(y,M)$ for any $y\in\mathcal{C}_H(x)$. 

\begin{equation}\label{2times}
\varphi_H(x)=\sup_{y\in\mathscr{C}_H(x)}d(x,\mathscr{C}_H(y))\leq \sup_{y\in\mathscr{C}_H(x)} d(x,M)+D(M,\mathcal{C}_H(y))\leq2d(x,M).
\end{equation}

$\bf (iv)\Rightarrow (ii)$\&$\bf (iii)\Rightarrow (i)$. If $\varphi_H(x)<\infty$ for some $x\in X$ then any closed convex invariant subset  of $\mathcal{C}_H(x)$ intersects the closed ball $\overline{B}(x,\varphi_H(x))$. By the analog of Banach-Alaoglu theorem for CAT(0) spaces  \cite[Theorem 14]{MR2219304}, every filtering family of $H$-invariant closed convex subsets has a non-trivial intersection. Zorn's Lemma yields an $H$-minimal subset in $\mathcal{C}_H(x)$.
\end{proof}

\begin{lemma}Let $ x,y\in X$. The function $H\mapsto d(x,\mathscr{C}_H(y))$ is measurable. In particular ${\bf\mathscr{C}(x)}=\{\mathscr{C}_H(x)\}$ is a subfield.
\end{lemma}

\begin{proof}It suffices actually to show that for any $x,y\in X$, the map $H\mapsto d(y,\mathscr{C}_H(x))$ is upper semicontinuous. Let us introduce some notations. For $x_1,x_2\in X$, let $m^1(x_1,x_2)$ be the midpoint of $x_1$ and $x_2$. We define by induction $m^{k}(x_1,\dots,x_{2^k})=m^1(m^{k-1}(x_1,\dots,x_{2^{k-1}}),m^{k-1}(x_{2^{k-1}+1},\dots,x_{2^k}))$. For any subset $A\subseteq X$, we denote by $C^k(A)$ the set $\{m^{k}(x_1,\dots,x_{2^k});\ x_1,\dots,x_{2^k}\in A\}$. Recall that the closed convex hull of $A$, $\overline{\Conv(A)}$ coincides with $\overline{\bigcup_{k\in\N}C^k(A)}$.

Assume $H_n\to H$ in $\SG$. Fix $\varepsilon>0$. There is $k\in\N$ such that $d(y,C^k(Hx))-d(y,\mathscr{C}_H(x))\leq\varepsilon/2$. Thus there are $h_1,\dots,h_{2^k}\in H$ such that $d(y,z)-d(y,\mathscr{C}_H(x))\leq\varepsilon/2$ where $z=m^{k}(h_1x,\dots,h_{2^k}x)$. For $n\in\N$, choose $h_i^n\in H_n$ converging to $h_i$ and set $z^n=m^{k}(h_1^nx,\dots,h_{2^k}^nx)$. Since $z^n\to z$, one has $\overline{\lim_{n\to\infty}}d(y,\mathscr{C}_{H_n}(x))\leq d(y,\mathscr{C}_H(x))+\varepsilon$ for all $\varepsilon>0$. This yields the desired semicontinuity.
\end{proof}
%Let $d$ be some metric on $G$ inducing the topology. Let us denote by $H^\varepsilon$ the closed $\varepsilon$-neighborhood of $H$ in $G$. Since for any $x\in X$, $\mathscr{C}_H(x)=\cap_{n}\overline{\Conv(H^{1/n}x)}$, it suffices to show that $H\mapsto d(x_H,\Conv(y_H))$ is measurable. Let us write $G=\{g_i\}_{i\in\N}$. We define $n_1(H)=\inf\{i\in\N;\ g_i\in H\}$ and by induction $n_{k+1}(H)=\inf\{i>n_k(H);\ g_i\in H\}$. Those are Borel functions of $H\in\SG$. Now, for $k\in\N$, $(g_{n_k(H)}y_H)_{H\in\SG}$ is a section of $\bf X$. Define $C^1(Hy_H)=\{g_{n_k(H)}y_H;\ k\in\N\}$ and by induction set $C^{l+1}(Hy_H)$ to be the set of midpoints of points in $C^{l}(Hy_H)$ (not necessarily distincts). Then $\mathscr{C}_H(y_H)$ is the closure of $\cup_{l\in\N}C^{l}(Hy_H)$. Since the midpoint of two sections yields a new section, one has that $H\mapsto d(x_H,\mathscr{C}_H(y_H))$ is measurable. \end{proof}

\begin{lemma}\label{varphi}For every $x\in X$, the function $H\mapsto\varphi_H(x)$ is measurable.
\end{lemma}

\begin{proof}Let $({ x}_n)$ be a dense countable family of $X$ then $\varphi_H(x)=\sup_{n}d(x_H,\mathscr{C}_H(x_n))$.
\end{proof}

\begin{lemma}\label{mini}Assume that for almost all $H\in\SG$, $H$ has a minimal invariant closed convex subset. For $H\in\SG$, let $Z_H$ be the union of closed convex minimal $H$-invariant subsets of $X_H$. Then ${\bf Z}=\{Z_H\}_{H\in\SG}$ is a $G$-invariant subfield of ${\bf X}_0$.
\end{lemma}

\begin{proof}It suffices to show that for any $x\in X$, the function $H\mapsto d(x,Z_H)$ is measurable. Fix a countable dense subset $\{x_n\}$ of $X$. This measurability is a consequence of the fact that 
\[d(x,Z_H)=\lim_{k\to\infty}\inf\{d(x,x_n);\ \varphi_H(x_n)<1/k\}.\]

To prove this equality, choose $x_n$ such that $d(x_n,z)<1/2k$ where $z$ is the projection of $x$ on $Z_H$. Since $\varphi_H(x_n)\leq 2 d(x_n,Z_H)$ (equation \eqref{2times}) then $d(x,x_n)<1/k+d(x,Z_H)$ and $d(x,Z_H)\geq\lim_{k\to\infty}\inf\{d(x,x_n);\ \varphi(x_n)<1/k\}$. 

Now, for any $n\in \N$, if $\varphi(x_n)<1/k$ then $d(x_n,Z_H)<1/k$. Let $z$ be the projection of $x_n$ on $Z_H$. We have $d(x,Z_H)\leq d(x,x_n)+d(x_n,z)\leq d(x,x_n)+1/k$. This yields the reverse inequality.
\end{proof}

\begin{proof}[Proof of Theorem \ref{IRSdense}]Thanks to the ergodic decomposition, it suffices to deal with the ergodic case. Consider ${\bf X}_0$, the field of CAT(0) spaces with constant fiber $X$ over $(\SG,\mu)$ on which $G$ acts. We apply  \cite[Proposition 8.11]{MR3044451} in case $X$ has finite telescopic dimension or  \cite[Theorem 5.1]{MR3163023} in case $X$ is proper with Tits boundary of finite dimension\footnote{The result is stated for proper CAT(0) spaces of finite dimension but  the finite dimension assumption is used only for the boundary and not for the space itself.} and get either a $G$-equivariant map $\SG\to\partial X$ or a minimal $G$-invariant subfield $\mathbf{X}$. The first possibility is ruled out by Lemma \ref{boundary map}. Thus, we consider the second one. 

Choose a point $x$ of $X$. Since $\{H\in\SG,\ \varphi_H(x)=+\infty\}$ is $G$-invariant and $G\action\SG$ is ergodic the measurable function (Lemma \ref{varphi}) $H\mapsto\varphi_H(x)$ is essentially constant equal to $+\infty$ or almost surely $H$ has a minimal closed convex subset in $X_H$. If $\varphi_H(x)=+\infty$, the intersection of boundaries of convex $H$-invariant subsets of $X_H$ has a canonical center $\xi_H$, which yields an equivariant measurable map $\SG\to\partial X$ contradicting Lemma \ref{boundary map}. Hence almost surely $X_H\cap Z_H\neq\emptyset$.  Since the intersection of two subfields is still a subfield, Lemma \ref{mini} and the minimality assumption on $\mathbf{X}$ implies almost surely $X_H=X_H\cap Z_H$. Applying Lemma \ref{convexfield}, we get $X_H=X$ almost surely. Thus $X=M_H\times T_H$ and this product has to be trivial since $X$ is irreducible. That is $M_H$ or $T_H$ is reduced to a point. If $M_H$ is point then all points are $H$-invariant and since $G\action X$ is faithful, this means $H=\{e\}$ and $\mu=\delta_e$. This contradicts the fact that the IRS is not trivial. Thus $T_H$ is point and this means $H\action X$ is minimal.

It remains to show that almost surely $H$ has no fixed point at infinity. Consider the push-forward of the measure $\mu$ under the map $H\mapsto H'=\overline{[H,H]}$. It yields a new IRS and one can apply what we did above and in particular, $H'$ acts minimally. But if $H$ has a fixed point at infinity, $H'$ stabilizes any horoball centered at this point and this contradicts the minimality of the action of $H'$ on $X$. 
\end{proof}

\begin{proof}[Proof of Theorem \ref{product case}]
Decomposing $\SG$ in  ergodic components, we may assume that $\mu$ is ergodic.
We first prove that almost surely $H\in\SG$ does not fix a point in $\partial X$. Since $\partial X$ is the spherical join $\partial X_1*\dots*\partial X_n$, if $H$ fixes a point in $\partial X$ then there is $i$ such that $H_i$ fixes a point in $\partial X_i$ where $H_i$ is the image of $H$ under the projection $G\to G_i$. This contradicts Theorem \ref{IRSdense} applied to the IRS $\mu_i$ on $G_i$.

Let $Z_H$ be the closed convex hull of $H$-minimal subsets, which is almost surely not empty thanks to the previous paragraph.  We claim that $Z_H=X$ and we prove it by an induction on $n$. The case $n=1$ follows from Theorem  \ref{IRSdense}. Assume the result holds for $n-1\geq1$. Since $H_n\action X_n$ is minimal, the projection of $Z_H$ to $X_n$ is $X_n$ itself. Let $\widehat{X_n}$ be a notation for $X_1\times \dots\times X_{n-1}$. Fix $x\in X_n$ and denote by $Z_H^x$ be the fiber over $\{x\}$ under the projection $X\to X_n$. This is a closed convex non-empty subspace of $\widehat{X_n}$. Observe that for any $g\in G$, $gZ_H^x=Z_{gHg^{-1}}^{g_nx}$ where $g_n$ is the $n$-th coordinate of $g$. In particular for $g\in \widehat{G_n}=G_1\times\dots\times G_{n-1}$, $gZ_H^x=Z_{gHg^{-1}}^{x}$ and $H\mapsto Z_H^x$ is a $\widehat{G_n}$-equivariant map. As $\widehat{G_n}$ acts minimally on $\widehat{X_n}$, it follows from Lemma \ref{convexfield} that $Z_H^x=X_1\times\dots\times X_{n-1}$. Since it works for any $x\in \widehat{X_n}$, one has that $Z_H=\widehat{X_n}\times X_n$. That is $Z_H=X$. Now, thanks to the uniqueness of the decomposition of $X$ has a product  of irreducible spaces (\cite[Proposition 6.1]{MR2558883}\&\cite[Theorem 5.1]{MR2574740}) one has $Z_H=M_H\times T_H$ with $M_H=X_1\times \dots\times X_k$ and  $T_H=X_{k+1}\times\dots\times X_n$ after a possible reordering of the $X_i$'s, where $M_H$ is an $H$-minimal subset and the action of $H$ on $T_H$ is trivial. By ergodicity, the number $k$ and the permutation are independent of $H$. Now, by definition, $H$ acts trivially on $T_H$ thus on $X_{k+1}\times\dots\times X_n$. But our hypothesis that $\mu_i$ is not trivial and the faithfulness of  $G_{i}\action X_i$ imply that $k=n$. That is $H$ acts minimally.
\end{proof}

\textbf{Acknowledgements} Y.G. is greatfull to the hospitality of the math department at the University of Utah as well as support from Israel Science Foundation grant ISF 441/11 and U.S. NSF grants DMS 1107452, 1107263, 1107367 ``RNMS: Geometric structures And Representation varieties" (the GEAR Network).

B.D. is supported in part by Lorraine Region and Lorraine University.

%\subsection{Zariski density}
%
%\begin{proof}[Proof of Theorem \ref{Zariski density}] Let $X$ be the symmetric space or Bruhat-Tits building associated to $G$. The action $G\action X$ is faithful, continuous and geometrically dense. If $\mu$ is a non-trivial IRS then almost surely $H\in\SG$ acts geometrically densely on $X$ and \cite[Proposition 2.8]{MR2574741} shows $H$ is Zariski-dense.
%\end{proof}

\bibliographystyle{halpha}
\bibliography{biblio.bib}

\end{document}